\documentclass[11pt]{article}
\usepackage{amssymb}
\usepackage{amsmath}
\usepackage{amsthm}
\usepackage{graphicx}
\usepackage[utf8]{inputenc}
\usepackage{amsfonts}
\usepackage{graphicx}
\usepackage{xcolor}
\usepackage{breqn}
\usepackage[utf8]{inputenc}
\usepackage{float}
\usepackage{hyperref}
\usepackage{cite}
\def\bc{\begin{center}}
	\def\ec{\end{center}}

\def\s2c{\vskip 2cm}
\def\bt{\begin{Theorem}}
	\def\et{\end{Theorem}}
\def\bd{\begin{Definition}}
	\def\ed{\end{Definition}}
\def\bl{\begin{Lemma}}
	\def\el{\end{Lemma}}
\def\be{\begin{Example}}
	\def\ee{\end{Example}}
\def\bcor{\begin{Corollary}}
	\def\ecor{\end{Corollary}}
\def\br{\begin{Remark}}
	\def\er{\end{Remark}}

\def\mysection{\setcounter{equation}{0}\section}
\newtheorem{Lemma}{Lemma}[section]
\newtheorem{Theorem}[Lemma]{Theorem}
\newtheorem{Definition}[Lemma]{Definition}
\newtheorem{Proposition}[Lemma]{Propostion}
\newtheorem{Corollary}[Lemma]{Corollary}
\newtheorem{Remark}[Lemma]{Remark}
\date{}
\title{Analysis of Diffusive Size-Structured Population Model and Optimal Birth Control}
\author {Manoj Kumar, Syed Abbas$^*$\\
School of Basic Sciences,\\
Indian Institute of Technology Mandi,\\	Kamand (H.P.) - 175005, India
	\\$^*$Corresponding author's email :  sabbas.iitk@gmail.com, abbas@iitmandi.ac.in}
	
\begin{document}
	\maketitle
	\author
	\noindent {\bf Abstract} :
This work addresses the optimal birth control problem for invasive species in a  spatial environment. We apply the method of semigroups to qualitatively analyze a size-structured population model in which individuals occupy a position in a spatial environment. With insect population in mind, we study the optimal control problem which takes fertility rate as a control variable. With the help of adjoint system, we derive optimality conditions. We obtain the optimality conditions by fixing the birth rate on three different sets. Using Ekeland's variational principle, the existence, and uniqueness of optimal birth controller to the given population model which minimizes a given cost functional is shown. Outcomes of our article are new and complement the existing ones.

\mysection{Introduction}

Modeling long-term evolutionary processes in the natural world is a challenging task because individuals in a population usually have a complex behavioral pattern. Furthermore, structuring natural populations in terms of size or developmental stages enhances the complexities of evolutionary models. Also, it is evident that the optimal behaviour and life characteristics of an individual should gradually alter with maturation and progression through different developmental stages or different sizes. Due to more complexities involved in structured population models as compared to unstructured population models, existing models are less developed.

 Mathematical models of population dynamics incorporating size-structure (size is a continuous variable which shows physiological or statistical characteristics of individuals, for example, size can be age, mass, diameter, length and maturity, and so on) have an extensive history. The age-structured models by Sharpe and Lotka \cite{Lotka1, Lotka2} are the earliest. Webb \cite{auger2008structured} discussed various size-structured population models with spatial structure and their analysis involves semigroup of linear and nonlinear operators. Method of semigroups considers an evolving size-structured population model as a dynamical system in an abtract space. Using semigroup of operators, N. Kato \cite{Kato2016} also studied a size-structured population model with diffusion term.

Control in a size-structured population model is the process of forcing a population through a controller to obtain a certain behaviour on it \cite{MR1797596,MR1702849}. The objective in these types of control problems is to obtain a controller that is optimal in the sense that it takes minimum cost to reduce vermin population or to increase total harvest. So, the task in these control problems is to either minimize a cost functional or maximize total harvest. For size-structured population models, these problems can be categorized either as optimal harvesting problems (optimal harvesting of natural or farmed populations such as fish or plants) or optimal control problems (optimal control of vermin or pest population).

There is enough literature available on the optimal control problems which take fertility rate as a control variable \cite{ MR1797596, MR3595204,MR2863964, MR1027051, MR1043120,NRN1,NRN2,NRN3,NRN4,NRN5,NRN6,NRN8} and also on optimal harvesting problems \cite{MR3703602,NRN9,MR2520362,MR937160}. Brokate \cite{NRN7} studied Pontryagin's principle for a general problem of optimal cotrol for an age structured population model. Also recent work on optimal behavioural strategies in structured populations is discussed in \cite{RN29}.

Motivated by the work of N. Kato \cite{Kato2016}, in this paper, we establish important estimates on mild solutions in terms of fertility rates and initial population distributions. The main focus of the N. Kato's work is to show the existence and uniqueness of mild solution, but in this work our main focus is to study the optimal birth control problem. To study optimal birth control problem, we need some estimates on mild solution and on the solution of adjoint equation. Rong et al. \cite{MR3595204} studied the least cost size and least cost deviation problem for a nonlinear size-structured population model which takes the fertility rate of vermin population as a control variable. Ze-Rong et al. in \cite{MR2863964} studied the optimal birth control problem without the diffusion term.
The technique used to find the optimality conditions is quite different from the technique used by \cite{MR3595204} and \cite{MR2863964}. The technique used by \cite{MR3595204} and \cite{MR2863964} are tangent-normal cone technique, but in our case we construct three sets and by fixing birth rate on each set, we obtain the optimality conditions. Using Ekeland's variational \cite{NRN10} principle and under some some assumptions, we show the existence and uniqueness of optimal controller as well.

This paper is divided into five sections. In section 2, we formulated the abstract form of our model with some assumptions. In section 3, we discuss some preliminary results which we will require in other sections.  In section 4, we discuss existence and uniqueness results and also derived some estimates on mild solution. The last section is devoted to the study of optimal birth control problem.

\section{Model Formulation}
Let us consider an evolving size-structured population living in a habitat $\Omega \subset \mathbb{R}^{n}$ with smooth boundary $\partial \Omega $. The spatial movement of individuals is controlled by diffusion process. Let $p(s,t,x) $ be the population density of individuals which depends on size variable $s \in [0,s_{f}]$, time $t \in[0,T]$ and spatial variable $x \in \Omega $. We assume that size of individuals increases in same manner and also individuals do not move out of domain $\Omega$ through $ \partial \Omega $. Integrating population density w.r.t. size and spatial variable at any time $t$ in the size interval $(s_{1},s_{2})$ and in a subset ${\Omega}_{1} \subset \Omega$ will give the total population $$ P(t)=\int_{s_{1}}^{s_{2}} \int_{\Omega_{1}} p(s,t,x) dx ds  .$$ Let $\gamma(s,t)$ be size and time dependent growth rate of individuals (growth rate here means a change in size with respect to time),  $\mu (s,t,x)$, $\beta (s,t,x)$ are mortality and reproduction rates respectively of size $s$ individuals at time $t$ in position $x$. Let $f(s,t,x)$ and $C(t,x)$ respectively be the number of $s$-size and zero size individuals coming from outside the domain $\Omega$ and let $r(s,t,x)$ be ratio of female individuals. For our convenience, we will use the following notations in our article $\Omega_{T} = (0,T) \times \Omega $, $\Omega_{s}=(0,s_{f}) \times \Omega $, $\Omega_{Ts}=(0,s_{f}) \times (0,T) \times \Omega$ and $\Sigma_{Ts}=(0,s_{f}) \times (0,T) \times \partial \Omega$. Now, size structured population model with diffusion can be formulated as \\

\begin{equation} \label{1.1}
\begin{cases}
  \frac{\partial p(s,t,x)}{\partial t} + \frac{\partial }{\partial s}(\gamma(s,t)p(s,t,x))= k \Delta p(s,t,x) - \mu (s,t,x)p(s,t,x)+ f(s,t,x) \quad \text{in} \quad \Omega_{Ts} \\
 \gamma(0,t)p(0,t,x)=C(t,x)+ \int_{0}^{s_{f}} r(s,t,x) \beta(s,t,x)p(s,t,x)ds \quad \text{in} \quad \Omega_{T} \\
 \frac{\partial p}{\partial \nu}(s,t,x)=0 \quad \text{in} \quad \Sigma_{Ts}, ~~~~~~~
 p(s,0,x)=p_{0}(s,x) \quad \text{in} \quad \Omega_{s}.\\
\end{cases}
\end{equation}
We assume $\beta(s,t,x)$ to be zero for those individuals who are not participating in the reproduction process. Constant $k$ is a diffusion coefficient and $\frac{\partial }{\partial \nu}$ denotes the derivative in the direction of outward normal to the boundary $\partial \Omega$. The homogeneous Neumann boundary condition shows that there is no flux of the population through the boundary of $\Omega$ and $p_{0}$ is the initial population density. Also if we take the Dirichlet boundary condition i.e. $p(s,t,x)=0$ on $\Sigma_{Ts}$ would mean individual will die as soon as they reached the boundary, which does not make sense for most of the species. So, the Neumann boundary condition is an appropriate choice. \\
 Let 
 \begin{equation} \label{n2.2}
  \mathcal{U} = \lbrace u \in L^{\infty}(\Omega_{Ts})~ |~ \phi_{l}(s,t,x) \le u(s,t,x) \le \phi_{m}(s,t,x) ~ a.e ~ \text{in} ~ \Omega_{Ts} \rbrace
  \end{equation} 
where $\phi_{l}$ and $\phi_{m}$ are non negative functions lies in $L^{\infty}(\Omega_{Ts})$. Let $p^{\beta}(s,t,x)$ be the solution of (\ref{1.1}), then our problem can be formulated as
\begin{equation} \label{min_1}
\text{minimize}~ J(\beta)= \int_{0}^{T} \int_{0}^{s_{f}} \int_{\Omega} \left[ p^{\beta}(s,t,x)-\frac{1}{2} \rho (\beta(s,t,x)^{2}) \right] ds dt dx
\end{equation}
subject to $\beta \in \mathcal{U}$ (\cite{MR2863964} have also considered similar cost functional but without diffusion term). Here, $\rho$ is a positive constant which is the weight factor of the cost to execute the control. $J(\beta)$ can be considered as an  energy function and our task is to show the existence of birth function $\beta$ which minimizes $J(\beta)$. It is clear from the energy function that if the cost of implementing control is high, then the fertility rate of invasive species will be lower. So, we want to minimize $J(\beta)$ given in (\ref{min_1}) subject to (\ref{1.1}). Some authors (\cite{MR3595204, MR1797596, MR1702849}) have also taken energy function as
\begin{dmath}
J(\beta)= \int_{0}^{T} \int_{0}^{s_{f}} \int_{\Omega} \left[ p^{\beta}(s,t,x)+\frac{1}{2} \rho (\beta(s,t,x)^{2}) \right] ds dt dx .
\end{dmath}
Same analysis can be applied to this energy function with some minor modifications.

\subsection{Abstract Formulation}
This model without spatial diffusion can be analyzed without converting it into abstract form, but to handle spatial diffusion we will convert this model into abstract form \cite{Kato2016} in an appropriate Banach space. \\
Let $A$ be realization of Laplacian in $L^{q}(\Omega)$, $q$ in $ (1, \infty)$ with the boundary conditions of Neumann type  i.e.
$$ D(A) = \left \{ v \in W^{2,q}(\Omega) \mid \frac{\partial v}{\partial \nu} = 0 \quad \text{a.e} \quad \text{on} \quad \partial \Omega \right \} $$
$$A\psi = k \Delta \psi \quad \text{for} \quad \psi \in D(A). $$ \\
It is well known that
$ A$ will generate an analytic semigroup $S(t), t \ge 0$ and there exist a positive constant $C$ and $\omega \in \mathbb{R}$, such that $\| S(t) \psi \|_{L^{q}(\Omega)} \le C \text{e}^{\omega t} \| \psi \|_{L^{q}(\Omega)}~ \text{for}~ \psi \in L^{q}(\Omega) $.\\
  For $\psi \in L^{q}(\Omega)$, let us define the operators $M$ and
  $B$ which are bounded linear  by \\
$$ [M(s,t)\psi](x) = \mu(s,t,x)\psi(x), \quad
 [B(s,t)\psi](x) = \beta(s,t,x)\psi(x)$$   $$ \text{also}~ [C(t)](x)=C(t,x),~ [f(s,t)](x)=f(s,t,x) ~\text{and}~ [p_{0}(s)](x)=p_{0}(s,x).$$
 Note that
 \begin{dmath}
 \| M \|_{L^{\infty}(S_{Tf};\mathcal{L}(Z))} \le \| \mu \|_{L^{\infty}(\Omega_{T_{s}})}, ~  \| B \|_{L^{\infty}(S_{Tf};\mathcal{L}(Z))} \le \| \beta \|_{L^{\infty}(\Omega_{T_{s}})}.
 \end{dmath}
 Above inequalities will be useful while proving existence and uniqueness of optimal control.\\
Let $Z$ be a Banach space with norm $\| \|. ~$ In abstract form our model reduces to
\begin{equation} \label{1.2}
\begin{cases}
 \frac{\partial p}{\partial t}(s,t) + \frac{\partial}{\partial s}(\gamma(s,t)p(s,t)) = [A - M (s,t)]p(s,t) + f(s,t), \quad (s,t) \in \overline{S}_{T_{f}} = [0,s_{f}] \times [0,T] \\
 \gamma(0,t)p(0,t)=C(t) + \int_{0}^{s_{f}} r(s,t) B(s,t)p(s,t)ds, \quad t \in [0,T]\\
   p(s,0)=p_{0}(s), \quad s \in [0,s_{f}].
\end{cases}
\end{equation}

Our main task here is to analyze $p(s,t)$ the $Z$ valued function, which describes the distribution of population at time $t$ with respect to size $s$.\\
We need the following assumptions to do qualitative analysis of our model:
\begin{itemize}
\item[(A1)] $\gamma \colon \overline{S}_{T_{f}} \mapsto [0, \infty) $ is continuously differentiable w.r.t both $s$ and $t$. Also $\gamma(s,t)$ is positive on $S_{T_{f}}$ i.e. size always increases after crossing initial size upto final size $s_{f}$.
We assume that the growth rate is non-negative and for each $ t \in [0,T] $, we consider the following cases:\\
(a) \quad $\gamma(0,t) >0$ and $\gamma(s_{f},t)>0$ \quad
(b) \quad $\gamma(0,t) >0$ and $\gamma(s_{f},t)=0$ \\
(c) \quad $\gamma(0,t) =0$ and $\gamma(s_{f},t)>0$ \quad
(d) \quad $\gamma(0,t) =0$ and $\gamma(s_{f},t)=0$. \\
\item[(A2)] $A$ is the infinitesimal generator of a $C_{0}$ semigroup $\lbrace S(t)~  |~  t \ge 0 \rbrace $ in $Z$ and $\| S(t) \psi \|_{Z} \le \tilde{C} \| \psi \|_{Z}$, where $\tilde{C}$ is a generic constant.
\item[(A3)] $ M, B \in L^{\infty}(S_{T_{f}} ; \mathcal{L}(Z))$, where $\mathcal{L}(Z)$ is the space of all bounded linear operators in $Z$.
\item[(A4)] $f \in L^{1}(S_{T_{f}};Z), C \in L^{1}(0,T; Z)$ and $p_{0} \in L^{1}(0,s_{f}; Z)$.
\item[(A5)] $0 < r(s,t,x) <1 ~ \forall~ (s,t,x) \in \Omega_{T_{s}}.$
\end{itemize}
We can extend the function $\gamma(s,t)$ on $\mathbb{R} \times [0,T]$ as follows: \\
Assume that the growth of individuals with size less than zero is same as size zero individuals. And the growth of individuals with size greater than $s_{f}$ is same as size $s_{f}$ individuals.

Due to assumption (A1), there exist a unique solution of the initial value problem
\begin{equation} \label{3.1}
\frac{d}{dt}s(t)=\gamma(s(t),t), \quad s(t_{0})=s_{0} \quad \text{where} \quad s_{0} \in \mathbb{R}.
\end{equation}
\begin{Remark} \label{rem}
The unique solution of  initial value problem (\ref{3.1}) is called characteristic curve of the system (\ref{1.1}).
\end{Remark}

Let us define $s(t)=\psi(t;t_{0},s_{0}), \ z_{0}(t)=\psi(t;0,0), \quad \text{and} \quad z_{1}(t)=\psi(t;T,s_{f}).$

Now, we will define the initial and final time using the idea of \cite{Kato2016} \\
\textbf{Case 1:}  For the case (A1)-(a), we have $z_{0}(t)>0$ for $t>0$ and $z_{1}(t)<s_{f}$ for $t<T$. Now, for $(s,t) \in \overline{S}_{T_{f}}$  satisfying $s \le z_{0}(t)$ there exists a unique $\tau_{0} \in [0,T]$ such that $\psi(t; \tau_{0},0)=s$. Existence of $\tau_{0}$ is due to unique solution of (\ref{3.1}) and the initial time $\tau_{0}(t,s)$ for $(s,t) \in \overline{S}_{T_{f}}$ is given by

\begin{equation} \label{ 3.2}
\begin{split}
 \tau_{0}(t,s) = \begin{cases}
\tau_{0}, & \text{if} \quad s \le z_{0}(t) \\
0, & \text{if} \quad s>z_{0}(t).
\end{cases}
\end{split}
\end{equation}
Similarly for $(s,t) \in \overline{S}_{T_{f}}$ satisfying $z_{1}(t) \le s$, there exists a unique $\tau_{1} \in [0,T]$ such that $\psi(t;\tau_{1},s_{f})=s$ and the final time is defined by
  \begin{equation} \label{3.3}
  \tau_{1}(t,s) = \begin{cases}
  \tau_{1}, & \text{if} \quad s \ge z_{1}(t) \\
  T, & \text{if} \quad s<z_{1}(t).
  \end{cases}
  \end{equation}
Similarly, we can define initial and final time for the other cases.

 \section{Preliminaries}
The definition of lower semicontinuity and other conditions for a function to achieve infimum are given in \cite{MR1797596}.\\
\textbf{Definition:}$~$  Let $Z$ be a topological space and $z_{0} \in Z$, then the function \\ $ f \colon Z \mapsto \mathbb{R} \cup \lbrace - \infty, \infty \rbrace $ is said to be \textbf{lower semi-continuous function} at $z_{0}$ if $$ \liminf_{z\rightarrow z_{0}} f(z) \ge f(z_{0}). $$
\begin{Proposition}( \cite{MR1797596}) \label{p1}
Let $Z$ be a Banach space which is also reflexive and $f \colon Z \mapsto \mathbb{R} \cup \lbrace -\infty , \infty \rbrace$ be  a lower semi-continuous convex function. If $K$ is a bounded, closed and convex subset of $Z$, then $f$ attains its infimum on $K$.
\end{Proposition}

\begin{Corollary} ( \cite{MR1797596}) \label{c1}
Let $Z$ be a reflexive Banach space and $f \colon Z \mapsto \mathbb{R} \cup \lbrace -\infty , \infty \rbrace$ be  a lower semi-continuous convex function such that $ \lim_{\| z \| \rightarrow z_{0}} f(z)= + \infty $, then $f$ attains its infimum on $Z$.
\end{Corollary}

The Ekeland variational principle is very useful to prove the existence of a minimum point in the absence of standard compactness condition \cite{NRN10}. The principle can be stated as follows:

\begin{Theorem} ( \cite{NRN10}) \label{t1}
 Let $(Z,d)$ be a complete metric space and $f \colon Z \mapsto (-\infty,\infty]$ be a lower semicontinuous function, bounded from below and nonidentically $+\infty$. Also let $\epsilon > 0$ and $z \in Z$ be such that $$ f(z) \le \inf \{ f(z) ~|~ z \in Z \} + \epsilon. $$ Then for any $\lambda > 0$ there exist $z_{\epsilon} \in Z $ such that $$ f(z_{\epsilon}) \le f(z), $$
$$ d(z_{\epsilon},z) \le \lambda), $$
$$ f(z_{\epsilon}) < f(z)+\epsilon \lambda^{-1} d(z_{\epsilon},z), ~~ \text{for any}~ z \in Z \setminus \{z_{\epsilon} \}. $$
\end{Theorem}
Now, we will state two important inequalities called Gronwall inequalities \cite{MR1797596}.
\begin{Lemma} ( \cite{MR1797596}) \label{l1}
 Let $f \colon [a,b] \mapsto \mathbb{R}~ (a,b \in \mathbb{R}, a<b)$ be a continuous function, $\varphi \in L^{\infty}(a,b)$ and $\lambda \in L^{1}(a,b), \lambda(t) \ge 0 ~\text{a.e.}~ t \in (a,b).$ If $$ f(t) \le \varphi(t)+ \int_{a}^{t} \lambda(s) f(s) ds, $$ for each $t \in [a,b]$, then $$ f(t) \le \varphi(t) + \int_{a}^{t} \varphi(s) \lambda(s) \text{exp}\left( \int_{s}^{t} \lambda(\tau) d \tau \right) ds, $$
for each $t \in [a,b].$
\end{Lemma}

\begin{Lemma} ( \cite{MR1797596}) \label{l2}
 If in addition to Lemma \ref{l1}, let $\varphi(t)=C$, where $C$ is constant, then $$ f(t) \le C \text{exp} \left( \int_{a}^{t} \lambda(s) ds \right). $$
\end{Lemma}




\mysection{Existence, Uniqueness and Estimates}
In this section, our task is to obtain important estimates on mild solution. These estimates will be important while studying optimal birth control problem. Firstly, we will give key steps to derive mild solution. Similar form of mild solution is also given by N. Kato \cite{Kato2016}, but without the ratio $r$. The key steps to derive mild solution are given for the convenience of readers.

Suppose $p(s,t)$ satisfies (\ref{1.1}) in strict sense and let

 $$ u(\zeta ;t,s) = \text{exp}\left(\int_{\tau_{0}}^{\zeta}
 \partial_{s} g (\psi(\eta ; t,s),\eta)) d \eta\right) p(s(\zeta),\zeta), \quad \text{where} \quad \psi(\zeta;t,s)= s(\zeta). $$

Differentiating $ u(\zeta ;t,s)$ with respect to $\zeta$, we get

\begin{equation} \label{4.1} \frac{d}{d \zeta}(u(\zeta ;t,s)) = [A - M(\psi(\zeta ;t,s),\zeta)] u(\zeta ;t,s) + \text{exp}\left(\int_{\tau_{0}}^{\zeta}
 \partial_{s} g (\psi(\eta ; t,s),\eta))d \eta\right) f(s(\zeta),\zeta).
\end{equation}
Using the variation of constant formula and after substituting $\zeta =t$, we obtain
$$ p(s,t) = Q(\tau_{0})  S(t- \tau_{0})p(0,\tau_{0}) + \int_{\tau_{0}}^{t} Q(\eta) S(t- \eta)[-M(s(\eta),\eta)p(s(\eta),\eta))+f(s(\eta),\eta)] d \eta,
$$
where $$Q(\tau_{0}) = \text{exp} \left( - \int_{\tau_{0}}^{t} \partial_{s} g (\psi(\eta ; t,s),\eta) d\eta \right) .$$
For the case (A1)-(a) and (A1)-(b), $\gamma(0,t) >0$, so $p(0,t)$ is defined by
\begin{equation} \label{eqn 4.2}
p(0,t) =\frac{1}{\gamma(0,t)}\left[C(t) + \int_{0}^{s_{f}} r(s,t) B(s,t)p(s,t)ds \right] \quad \text{for} \quad t \in (0,T).
\end{equation}
Also, for the case (A1)-(a) and (A1)-(b), we have
$$\tau_{0}(t,s)=\tau_{0} \quad \text{for} \quad s \in (0,z_{0}(t)) \quad \text{and} \quad \tau_{0}(t,s) = 0 \quad \text{for} \quad s \in (z_{0}(t),s_{f}).$$

Moreover, in this case for a.e $t \in (0,T)$, $p(s,t)$ is given by
\[
p(s,t) = \begin{cases}

Q(\tau_{0})  S(t- \tau_{0})p(0,\tau_{0}) + \int_{\tau_{0}}^{t} Q(\eta) S(t- \eta)[-M(s(\eta),\eta)p(s(\eta),\eta))+f(s(\eta),\eta)] d \eta
 &  ,  s \in (0,z_{0}(t))

   \\

Q(0)  S(t)p_{0} + \int_{0}^{t} Q(\eta) S(t- \eta)[-M(s(\eta),\eta)p(s(\eta),\eta))+f(s(\eta),\eta)] d \eta
& ,  s \in (z_{0}(t),s_{f}).

\end{cases}
\]
Note that $p_{0}=p_{0}(\psi(0;t,s))$ and for the case (A1)-(c) and (A1)-(d), there is no reproduction rate, so for these cases, we have
$$p(s,t) =Q(0)  S(t)p_{0} + \int_{0}^{t} Q(\eta) S(t- \eta)[-M(s(\eta),\eta)p(s(\eta),\eta))+f(s(\eta),\eta)] d \eta \quad \text{for} \quad (s,t) \in S_{T_{f}}.$$

We define the above solution with (\ref{eqn 4.2}) as \textbf{mild solution}.

\begin{Remark} \label{r1}
 Suppose $\mu, \beta \in L^{\infty}(\Omega_{Ts})$ and assume that $\mu, \beta$ are non negative almost everywhere in $\Omega_{Ts}$. Let $f \in L^{1}(\Omega_{Ts}), f(s,t,x) \ge 0 ~\text{a.e.}~ (s,t,x) \in \Omega_{Ts} $  and $C \in L^{1}(\Omega_{T}), C(t,x) \ge 0 ~ \text{a.e.}~ (t,x) \in \Omega_{T}$ and also $p_{0} \in L^{1}(\Omega_{s}), p_{0}(s,x) \ge 0 ~ \text{a.e.} ~ (s,x) \in \Omega_{s}$. If the semigroup $\{ S(t) ~|~ t \ge 0 \}$ satisfies $S(t)Z_{+} \subset Z_{+}$, we have existence of  $p \in L^{\infty}(0,T;L^{1}(0,s_{f};Z_{+}))$ and $p(s,t) \in Z_{+} ~\text{a.e.}~ (s,t) \in S_{Tf}$, where $Z_{+}$ is a positive cone with vertex $0$.
\end{Remark}

Let $C_{\psi}^{k}(S_{T_{f}} ;Z)$ be the collection of those functions which are $Z$-valued and $k$ times continuously differentiable along the characteristic curve $\psi$. Suppose $D_{\psi}p(s,t)$ is the derivative along the characteristic $\psi$  $$ D_{\psi}p(s,t):= \lim_{h \to 0} \frac{[p(\psi(t+h;t,s),t+h)-p(s,t)]}{h}. $$ Then, it is easy to see that $$D_{\psi}p(\psi(\eta;t,s),\eta)=\frac{d}{d \eta}p(\psi(\eta;t,s),\eta).$$

 \begin{Theorem} \label{mt1}
  For the \textbf{Case (A1)-(a)} and \textbf{Case (A1)-(b)}, any mild solution $p \in L^{\infty}(0,T; L^{1}(0,s_{f};Z)) $ of (\ref{1.1}) is continuously differentiable along almost every characteristic $\psi$ and satisfies
 \begin{equation} \label{4.3}
 \begin{cases}
  D_{\psi}p(s,t) = A p(s,t) - \partial_{s}\gamma(s,t)p(s,t) - M (s,t)p(s,t)+f(s,t)  \quad \text{a.e} \quad (s,t) \in S_{T_{f}} \\
 \gamma(0,t)p(0,t) = C(t) + \int_{0}^{s_{f}} r(s,t)B(s,t)p(s,t)ds, \quad \text{a.e} \quad t \in (0,T) \\
$$ p(s,0 = p_{0}(s), \quad \text{a.e} \quad s \in (0,s_{f})
 \end{cases}
\end{equation}
where $p(s,0)$ and $p(0,t)$ are given by the following limits:
$$ p(s,0) = \lim_{\eta \to 0^+}p(\psi(\eta;0,s),\eta) \quad \text{a.e} \quad s \in (0,s_{f}), $$
$$ p(0,t) = \lim_{\eta \to t}p(\psi(\eta;t,0),\eta) \quad \text{a.e} \quad t \in (0,T), $$ and in \textbf{case (A1)-(c)} and \textbf{case (A1)-(d)}
\begin{equation} \label{4.4}
\begin{cases}
 D_{\psi}p(s,t) = A p(s,t) - \partial_{s}\gamma(s,t)p(s,t) - M (s,t)p(s,t)+f(s,t)  \quad \text{a.e} \quad (s,t) \in S_{T_{f}} \\
 p(s,0)=p_{0}(s), \quad \text{a.e} \quad s \in (0,s_{f}).
 \end{cases}
 \end{equation}

Conversely, if $p \in L^{\infty}(0,T; L^{1}(0,s_{f};Z)) \cap C_{\psi}^{1}(S_{Tf};Z) $ satisfies (\ref{4.3}) and (\ref{4.4}), then $p$ is a mild solution of (\ref{1.2}). \\
\end{Theorem}
\begin{proof}
 The proof of necessary part is easy and follows just by using the fact that $A $ generates an analytic semigroup. Conversely, suppose that $p \in L^{\infty}(0,T; L^{1}(0,s_{f};Z)) \cap C_{\psi}^{1}(S_{Tf};Z) $ and satisfies (\ref{4.3}). So, $$ D_{\psi}p(s,t) = A p(s,t) - \partial_{s}\gamma(s,t)p(s,t) - M (s,t)p(s,t)+f(s,t)  \quad \text{a.e} \quad (s,t) \in S_{T_{f}} $$ because
$$D_{\psi}p(\psi(\eta;t,s),\eta)=\frac{d}{d \eta}p(\psi(\eta;t,s),\eta).$$
Differentiating $p$ with respect to $\eta,$ we obtain 
\begin{dmath*}
\frac{d}{d \eta}p(\psi(\eta;t,s),\eta) = A p(\psi(\eta;t,s),\eta) - \partial_{s}\gamma(\psi(\eta;t,s),\eta)p(\psi(\eta;t,s),\eta) \\
- M (\psi(\eta;t,s),\eta)p(\psi(\eta;t,s),\eta)+f(\psi(\eta;t,s),\eta) 
\end{dmath*}
$$
 p(\psi(\tau_{0};t,s),\tau_{0})=p(0,\tau_{0}).$$
Therefore, by variation of constants formula, the solution starting from initial time $\tau_{0}$ with $\eta = t$ is given by $$ Q(\tau_{0})  S(t- \tau_{0})p(0,\tau_{0}) + \int_{\tau_{0}}^{t} Q(\eta) S(t- \eta)[-M(s(\eta),\eta)p(s(\eta),\eta))+f(s(\eta),\eta)] d \eta,$$
where   $ Q(\tau_{0}) = \text{exp} \left( - \int_{\tau_{0}}^{t} \partial_{s} g (\psi(\eta ; t,s),\eta) d\eta \right). $ Also $p(\psi(0;t,s),0)=p_{0}(\psi(0;t,s))$, so $p$ is a mild solution to (\ref{1.2}).
\end{proof}
One can observe that $s=\zeta = \tau_{0}(t,s),$ which implies
$s=\psi(t;\zeta,0)$ and using (\ref{3.1}), we have
$$ \frac{ds}{d \zeta} = -\gamma(0,\zeta) \text{exp} \left( -\int_{\zeta}^{t} \partial_{s} g (\psi(\eta;\zeta,0),\eta)d \eta \right). $$
Moreover $s=\phi = \psi(\zeta;t,s),$ which implies
$s= \psi(t;\zeta,\phi)$ and again using (\ref{3.1}), we have
$$ \frac{ds}{d \phi}= \text{exp} \left( \int_{\zeta}^{t} \partial_{s}\gamma(\psi(\eta;t,s),\eta)d \eta \right). $$
Suppose that the assumptions (A1)-(A5) holds, and suppose that $p_{1},p_{2} \in L^{\infty}(0,T;L^{1}(0,s_{f};Z))$ be mild solutions to (\ref{1.1}).
Then for $t \in [0,T]$
$$ \| p_{1}(\cdot,t)-p_{2}(\cdot,t) \|_{L^{1}(0,s_{f};Z)} \le \int_{0}^{z_{0}(t)} \| Q(\tau_{0}) S(t-\tau_{0})(p_{1}(0,\tau_{0})-p_{2}(0,\tau_{0})) \|_{Z}~ ds $$
$$ + \int_{0}^{z_{0}(t)} \int_{\tau_{0}}^{t} \| Q(\eta)S(t-\eta) M(s(\eta),\eta)(p_{1}-p_{2}) \|_{Z}~ d \eta ds + \int_{z_{0}(t)}^{s_{f}} \int_{0}^{t} \| Q(\eta)S(t-\eta) M(s(\eta),\eta)(p_{1}-p_{2}) \|_{Z}~ d \eta ds.$$     Using the above transformation $s = \zeta$ to first integral on R.H.S  and $s = \phi$ to second and third integral on R.H.S,we will get
$$ \| p_{1}(\cdot,t)-p_{2}(\cdot,t) \|_{L^{1}(0,s_{f};Z)} \le \tilde{C} \int_{0}^{t} \| p_{1}(\cdot,\eta)-p_{2}(\cdot,\eta) \|_{L^{1}(0,s_{f};Z)}~d \eta,$$
where $\tilde{C}$ is generic constant which depends on $ \| B \|_{L^{\infty}(S_{Tf};\mathcal{L}(Z))} , \| M \|_{L^{\infty}(S_{Tf};\mathcal{L}(Z))},T $ and the bound of semigroup $S(t)$. By using  Gronwall's inequality, we get
$$  \| p_{1}(\cdot,t)-p_{2}(\cdot,t) \|_{L^{1}(0,s_{f};Z)}=0,$$   which shows that the mild solution is unique.


\begin{Theorem} \label{mt2}
Suppose that the assumptions (A1)-(A5) hold and $p_{1}, p_{2} \in L^{\infty}(0,T;L^{1}(0,s_{f};Z))$ be mild solutions corresponding to the initial data $p_{01}$ and $p_{02}$ respectively. Then  $$ \| p_{1}(\cdot,t)-p_{2}(\cdot,t) \|_{L^{1}(0,s_{f};Z)} \le \tilde{C} \| p_{01}-p_{02} \|_{L^{1}(0,s_{f};Z)}, $$
where $\tilde{C}$ is a generic constant which depends on $ \| B \|_{L^{\infty}(S_{Tf};\mathcal{L}(Z))} , \| M \|_{L^{\infty}(S_{Tf};\mathcal{L}(Z))},T $ and the bound of the semigroup $\{ S(t) , t \ge 0 \}$.
\end{Theorem}
\begin{proof}
 For $t \in [0,T]$ and $p_{1}, p_{2} \in L^{\infty}(0,T;L^{1}(0,s_{f};Z)),$ we have
 \begin{eqnarray*}
\| p_{1}(\cdot,t)-p_{2}(\cdot,t) \|_{L^{1}(0,s_{f};Z)} & \le & \int_{0}^{z_{0}(t)} \| Q(\tau_{0}) S(t-\tau_{0})(p_{1}(0,\tau_{0})-p_{2}(0,\tau_{0})) \|_{Z}~ ds \\&+& \int_{z_{0}(t)}^{s_{f}} \| Q(0)S(t)(p_{01}-p_{02}) \|_{Z}~ ds
\\ &+& \int_{0}^{z_{0}(t)} \int_{\tau_{0}}^{t} \| Q(\eta)S(t-\eta) M(s(\eta),\eta)(p_{1}-p_{2}) \|_{Z}~ d \eta ds \\ &+& \int_{z_{0}(t)}^{s_{f}} \int_{0}^{t} \| Q(\eta)S(t-\eta) M(s(\eta),\eta)(p_{1}-p_{2}) \|_{Z}~ d \eta ds.
\end{eqnarray*}

Let us denote $$ I_{1}= \int_{0}^{z_{0}(t)} \| Q(\tau_{0}) S(t-\tau_{0})(p_{1}(0,\tau_{0})-p_{2}(0,\tau_{0})) \|_{Z}~ ds, $$
$$ I_{2}= \int_{z_{0}(t)}^{s_{f}} \| Q(0)S(t)(p_{01}-p_{02}) \|_{Z}~ ds, $$
$$ I_{3}= \int_{0}^{z_{0}(t)} \int_{\tau_{0}}^{t} \| Q(\eta)S(t-\eta) M(s(\eta),\eta)(p_{1}-p_{2}) \|_{Z}~ d \eta ds, $$
$$ I_{4}= \int_{z_{0}(t)}^{s_{f}} \int_{0}^{t} \| Q(\eta)S(t-\eta) M(s(\eta),\eta)(p_{1}-p_{2}) \|_{Z}~ d \eta ds. $$
Now, again using the transformation $s=\zeta= \tau_{0}(t,s)$ on $I_{1}$, $s=\phi=\psi(0;t,s)$ on $I_{2}$ and $s=\phi = \psi(\zeta;t,s)$ on $I_{3}$ and $I_{4}$, we get
$$ \| p_{1}(\cdot,t)-p_{2}(\cdot,t) \|_{L^{1}(0,s_{f};Z)} \le \tilde{C} \| p_{01}-p_{02} \|_{L^{1}(0,s_{f};Z)} + \tilde{C} \int_{0}^{t} \| p_{1}(\cdot,\zeta)-p_{2}(\cdot, \zeta) \|_{L^{1}(0,s_{f};Z)~ d\zeta},$$
where $\tilde{C}$ is generic constant which depends on $ \| B \|_{L^{\infty}(S_{Tf};\mathcal{L}(Z))} , \| M \|_{L^{\infty}(S_{Tf};\mathcal{L}(Z))},T $ and the bound of the  semigroup $S(t)$.
By using the Gronwall's inequality, we obtain    $$ \| p_{1}(\cdot,t)-p_{2}(\cdot,t) \|_{L^{1}(0,s_{f};Z)} \le \tilde{C} \| p_{01}-p_{02} \|_{L^{1}(0,s_{f};Z)}.$$
\end{proof}
In a similar manner, we can state the following result:
\begin{Theorem} \label{mt3}
Suppose that the assumptions (A1)-(A5) hold and $p^{\beta_{1}},p^{\beta_{2}} \in L^{\infty}(0,T;L^{1}(0,s_{f};Z))$ be solutions corresponding to the fertility rates $\beta_{1}$ and $\beta_{2}$ respectively. Then
\begin{equation}
 \left \| p^{\beta_{1}}(\cdot,t)-p^{\beta_{2}}(\cdot,t) \right \|_{L^{1}(0,s_{f};Z)} \le \tilde{C} \text{exp} \left( T \| {B}_{1}-{B}_{2} \|_{L^{\infty}(S_{Tf};\mathcal{L}(Z))} \right),
\end{equation}
where $\tilde{C}$ is a positive constant which depends on  $\| M \|_{L^{\infty}(S_{Tf};\mathcal{L}(Z))},T$ and the bound of the semigroup $\{ S(t) , t \ge 0 \}$. \\
Moreover, if $\tilde{C} \le \frac{C_{1}} {\| {B}_{1}-{B}_{2} \|_{L^{\infty}(S_{Tf};\mathcal{L}(Z))}}~ \forall~ {B}_{1},{B}_{2} \in \mathcal{U} $ and for some positive constant $C_{1}$, then
\begin{eqnarray}
 \left \| p^{\beta_{1}}(\cdot,t)-p^{\beta_{2}}(\cdot,t) \right \|_{L^{1}(0,s_{f};Z)} \le M_{1}\| {B}_{1}-{B}_{2} \|_{L^{\infty}(S_{Tf};\mathcal{L}(Z))},
 \end{eqnarray}
 where $M_{1}$ is a positive constant.
 \end{Theorem}

 \begin{proof} For $t \in [0,T]$ and $p^{\beta_{1}},p^{\beta_{2}} \in L^{\infty}(0,T;L^{1}(0,s_{f};Z)),$ we have

  \begin{eqnarray*}
  \| p^{\beta_{1}}(\cdot,t)-p^{\beta_{2}}(\cdot,t) \|_{L^{1}(0,s_{f};Z)} &\le & \int_{0}^{z_{0}(t)} \| Q(\tau_{0}) S(t-\tau_{0})(p^{\beta_{1}}(0,\tau_{0})-p^{\beta_{2}}(0,\tau_{0})) \|_{Z}~ ds  \\
    & + & \int_{0}^{z_{0}(t)} \int_{\tau_{0}}^{t} \| Q(\eta)S(t-\eta) M(s(\eta),\eta)(p^{\beta_{1}}-p^{\beta_{2}}) \|_{Z}~ d \eta ds \\&+& \int_{z_{0}(t)}^{s_{f}} \int_{0}^{t} \| Q(\eta)S(t-\eta) M(s(\eta),\eta)(p^{\beta_{1}}-p^{\beta_{2}}) \|_{Z}~ d \eta ds.
  \end{eqnarray*}
  Let us denote $$ I_{1} =  \int_{0}^{z_{0}(t)} \| Q(\tau_{0}) S(t-\tau_{0})(p^{\beta_{1}}(0,\tau_{0})-p^{\beta_{2}}(0,\tau_{0})) \|_{Z}~ ds  $$
  $$ I_{2} = \int_{0}^{z_{0}(t)} \int_{\tau_{0}}^{t} \| Q(\eta)S(t-\eta) M(s(\eta),\eta)(p^{\beta_{1}}-p^{\beta_{2}}) \|_{Z}~ d \eta ds $$
  $$ I_{3} = \int_{z_{0}(t)}^{s_{f}} \int_{0}^{t} \| Q(\eta)S(t-\eta) M(s(\eta),\eta)(p^{\beta_{1}}-p^{\beta_{2}}) \|_{Z}~ d \eta ds $$
  In the integral $I_{1}$, substituting $s = \zeta = \tau_{0}(t,s),$ we get
  $$I_{1} \le \int_{0}^{t} \left\| S(t-\zeta) \left( \int_{0}^{s_{f}} B_{1}(s(\eta),\eta)p^{\beta_{1}}(s(\eta),\eta) d \eta -
  \int_{0}^{s_{f}} B_{2}(s(\eta),\eta)p^{\beta_{2}}(s(\eta),\eta) d \eta \right) \right\|_{Z} d \zeta . $$
In \cite{Kato2016}, it has been shown that if $\beta_{2} \le \beta_{1}$, then $p^{\beta_{2}} \le p^{\beta_{1}}$. So, here we prove the other case. Without loss of generality assume that
$\beta_{2} \ge \beta_{1}$, then
$$\beta_{1}p^{\beta_{1}}-\beta_{2}p^{\beta_{2}} \le \beta_{1}p^{\beta_{1}}+\beta_{2}p^{\beta_{2}}-\beta_{1}p^{\beta_{2}}- \beta_{2}p^{\beta_{1}} = (\beta_{1}-\beta_{2})(p^{\beta_{1}}- p^{\beta_{2}}).$$
Therefore,
$$ I_{1} \le \tilde{C} \| {B}_{1}-{B}_{2} \|_{L^{\infty}(S_{Tf};\mathcal{L}(Z))} \int_{0}^{t} \left \| p^{\beta_{1}}(\cdot,\zeta)-p^{\beta_{2}}(\cdot,\zeta) \right \|_{L^{1}(0,s_{f};Z)} d \zeta .   $$
Similarly, using $s=\phi = \psi(\zeta;t,s)$, we can always find $\tilde{C}$ such that
$$I_{2}+I_{3} \le \tilde{C} \| {B}_{1}-{B}_{2} \|_{L^{\infty}(S_{Tf};\mathcal{L}(Z))} \int_{0}^{t} \left \| p^{\beta_{1}}(\cdot,\zeta)-p^{\beta_{2}}(\cdot,\zeta) \right \|_{L^{1}(0,s_{f};Z)} d \zeta . $$
 Therefore
 \begin{equation} \label{u4.7}
 \left \| p^{\beta_{1}}(\cdot,t)-p^{\beta_{2}}(\cdot,t) \right \|_{L^{1}(0,s_{f};Z)} \le C  + \tilde{C} \| {B}_{1}-{B}_{2} \|_{L^{\infty}(S_{Tf};\mathcal{L}(Z))} \int_{0}^{t} \left \| p^{\beta_{1}}(\cdot,\zeta)-p^{\beta_{2}}(\cdot,\zeta) \right \|_{L^{1}(0,s_{f};Z)} d \zeta,
   \end{equation}
 where $C$ is a positive constant and $\tilde{C}$ is a generic constant which depends on  $\| M \|_{L^{\infty}(S_{Tf};\mathcal{L}(Z))}, \ T$ and the bound of the semigroup $\{ S(t) , t \ge 0 \}$. Now using the Gronwall inequality, we get
 $$
  \left \| p^{\beta_{1}}(\cdot,t)-p^{\beta_{2}}(\cdot,t) \right \|_{L^{1}(0,s_{f};Z)} \le C \text{exp} \left( \tilde{C} T \| {B}_{1}-{B}_{2} \|_{L^{\infty}(S_{Tf};\mathcal{L}(Z))} \right).
$$
 Since $C$ is a positive constant, we can find another constant $C$ (just for convenience) such that
 $$ C \le C \| {B}_{1}-{B}_{2} \|_{L^{\infty}(S_{Tf};\mathcal{L}(Z))}. $$
 Using the estimate on $\tilde{C}$ and Gronwall's lemma, we get
 \begin{eqnarray*}
 \left \| p^{\beta_{1}}(\cdot,t)-p^{\beta_{2}}(\cdot,t) \right \|_{L^{1}(0,s_{f};Z)} \le C\| {B}_{1}-{B}_{2} \|_{L^{\infty}(S_{Tf};\mathcal{L}(Z))} e^{C_{1}T}.
 \end{eqnarray*}
 Taking $M_{1} = Ce^{C_{1}T}$, we obtain our desired result.
 \end{proof}

\begin{Remark}
Applying the same procedure as we followed to prove the uniqueness of mild solution to (\ref{1.1}), we can prove that the bound of mild solution depends on vital rates, inflow of zero and $s$-size individuals and the initial population distribution. So, if we assume inflow of zero and $s$-size individuals as constant and also vital rates (mortality, fertility and growth rate) as constant, then mild solution will be global mild solution in time variable.
\end{Remark}
Now, let us consider the following adjoint system to (\ref{4.3}) and (\ref{4.4}): \\
In case (A1)-(a) and (A1)-(c),
\begin{equation} \label{4.5}
\begin{cases}
D_{\psi} \phi (s,t)=-A^{*} \phi (s,t) + M^{*}(s,t) \phi (s,t) +[1- B^{*}(s,t)r(s,t) \phi (0,t)] + f^{*}(s,t) ~~ \text{a.e} ~~ (s,t) \in S_{Tf} \\
 \phi (s_{f},t)=\lim_{h \rightarrow +0} \phi(\psi(t-h;t,s_{f}),t-h)=0 ~~ \text{a.e} ~~ t \in (0,T) \\
 \phi(s,T)= \lim_{h \rightarrow +0} \phi (\psi(T-h;T,s),T-h)=0, ~~ \text{a.e} ~~ s \in (0,s_{f})
 \end{cases}
 \end{equation}
 and for case (A1)-(b) and (A1)-(d)
 \begin{equation} \label{4.6}
 \begin{cases}
 D_{\psi} \phi (s,t)=-A^{*} \phi (s,t) + M^{*}(s,t) \phi (s,t) +[1- B^{*}(s,t)r(s,t) \phi (0,t)] + f^{*}(s,t) ~~ \text{a.e} ~~ (s,t) \in S_{Tf} \\
 \phi (s_{f},t)=\lim_{h \rightarrow +0} \phi(\psi(t-h;t,s_{f}),t-h)=0 ~~ \text{a.e} ~~ t \in (0,T),

 \end{cases}
 \end{equation}
 where $\phi \in L^{\infty}(S_{Tf};D(A^{*})) \cap C_{\psi}^{1}(S_{Tf};Z^{*})$ is unknown. In adjoint system, $A^{*}$ is the adjoint operator of $A$ and will generate the adjoint semigroup $\lbrace S^{*}(t) | t \ge 0  \rbrace$ in the dual space $Z^{*}$ of $Z$, also $B^{*}$ and $M^{*}$ are adjoint operators for $B$ and $M$ respectively. $M^{*}$ and $B^{*}$ are defined by
 $$ [M^{*}(s,t) \phi](x) =\mu(s,t,x)\phi(x),~ [B^{*}(s,t) \phi](x) =\beta(s,t,x)\phi(x).  $$
 Note that
 \begin{dmath} \label{dest1}
 \| M^{*} \|_{L^{\infty}(S_{T_{f}};\mathcal{L}(Z^{*}))} \le \| \mu \|_{L^{\infty}(\Omega_{T_{s}})},~ \| B^{*} \|_{L^{\infty}(S_{T_{f}};\mathcal{L}(Z^{*}))} \le \| \beta \|_{L^{\infty}(\Omega_{T_{s}})} .
 \end{dmath}
 Now, our task is to convert the data $\phi (s_{f},t)=\phi(s,T)=0$ into the data $\phi (0,t)= \phi (s,0)=0 $ to relate system (\ref{4.5}) and (\ref{4.3}). So, the natural choice of transformation will be $ \overline{\phi}(s,t)=\phi(s_{f}-s,T-t) $. Our next job is to find the characteristic curve along which $\overline{\phi}$ satisfy an ordinary differential equation. Let $\overline{\gamma}(s,t)=\gamma(s_{f}-s,T-t)$ and $\overline{\psi}(t;t_{0},s_{0})$ be the solution to the system $$ \frac{ds(t)}{dt}=\overline{\gamma}(s(t),t), ~~ s(t_{0})=s_{0}, ~~ t \in [0,T]. $$ By following the idea of \cite{Kato2016}, one can easily show that $ D_{\overline{\psi}}\overline{\phi}(s,t)=-D_{\psi} \phi(s_{f}-s,T-t). $  Then applying the same procedure as used in Theorem \ref{mt1}, we can show the existence of solution to the dual system (\ref{4.5}) and (\ref{4.6}). \\
 Also the uniqueness of the solution to the adjoint system follows from the following estimates
 $$\| \phi(s,t) \|_{Z^{*}} \le C^{*} \int_{0}^{T} \|   B^{*}(\cdot,\zeta) r(\cdot,\zeta) \phi(0,\zeta) - r(\cdot,\zeta) \phi(0,\zeta) - f^{*}(\cdot,\zeta) \|_{L^{\infty}(0,s_{f};Z^{*})} d \zeta.$$

 \vspace{0.5cm}
 Following the same steps as followed in theorem \ref{mt3}, we will get the following estimates for dual problem:

 \begin{Theorem} \label{mt6}
 Suppose that the assumptions (A1)-(A5) hold and $\phi^{\beta^{1}}, \phi^{\beta_{2}}$ be solutions to dual problem corresponding to fertility rates $\beta_{1}, \beta_{2}$ respectively, then
\begin{dmath} \label{d4.10}
 \| \phi^{\beta_{1}}(s,t) - \phi^{\beta_{2}}(s,t) \|_{Z^{*}} \le M_{2}  \| B^{*}_{1} - B^{*}_{2} \|_{L^{\infty}(S_{T_{f}};\mathcal{L}(Z^{*}))} .
 \end{dmath}
 \end{Theorem}
 \begin{Remark}
 Using (\ref{dest1}) estimate in theorem \ref{mt6} can be written as
 $$| \phi^{\beta_{1}}(0,t,x)-\phi^{\beta_{2}}(0,t,x) | \le M_{2} \| \beta_{1} - \beta_{2} \|_{L^{\infty}(\Omega_{T_{s}})} ~ \forall t \in (0,T),~ x \in \Omega . $$
 \end{Remark}
 Let us define the concept of \textbf{weak solution} to (\ref{1.1}) as follows \\
 In case (A1)-(a) and case (A1)-(b), a function $p \in L^{1}(S_{Tf};Z)$ is called a weak solution to (\ref{1.1}) in the following sense
 $$ \int_{0}^{T} \int_{0}^{s_{f}} \int_{\Omega} p(s,t,x)(- D_{\psi} \phi (s,t,x)- k \Delta \phi(s,t,x)+ \mu(s,t,x) \phi(s,t,x) - r(s,t,x)\beta (s,t,x) \phi(0,t,x)) dx ds dt $$
 $$ = \int_{0}^{s_{f}} \int_{\Omega} p_{0}(s,x) \phi(s,0,x) dx ds + \int_{0}^{T} \int_{\Omega} C(t,x) \phi(0,t,x) dx dt + \int_{0}^{T} \int_{0}^{s_{f}} \int_{\Omega} f(s,t,x) \phi (s,t,x) dx ds dt $$
 where $\phi$ is any absolutely continuous function along almost every characteristic line and in case (A1)-(a) satisfies
 \begin{equation} \label{4.7}
 \begin{cases}
 \phi \in L^{\infty}(\Omega_{Ts}) \\
 D_{\psi} \phi +k \Delta \phi -\mu \phi + \beta \phi(0,\cdot,\cdot) \in L^{\infty}(\Omega_{Ts}) \\
 \phi (s_{f},t)=0 ~~ a.e~ t \in (0,T) \\
 \phi(s,T)=0 ~~ a.e ~ s \in (0,s_{f}),
 \end{cases}
 \end{equation}
 and in case (A1)-(b) satisfies
 \begin{equation} \label{4.8}
 \begin{cases}
 \phi \in L^{\infty}(\Omega_{Ts}) \\
 D_{\psi} \phi +k \Delta \phi -\mu \phi + \beta \phi(0,\cdot,\cdot) \in L^{\infty}(\Omega_{Ts}) \\
 \phi(s,T)=0 ~~ a.e ~ s \in (0,s_{f}).
 \end{cases}
 \end{equation}
 Denote set of all $ \phi $ satisfying (\ref{4.7}) by $\Phi_{1}$ and satisfying (\ref{4.8}) by $\Phi_{2}$. Since $\phi$ is absolutely continuous along almost every characteristic line, $\phi (s_{f},t)$ and $\phi(s,T)$ should be understood as
 $$ \phi(s_{f},t)= \lim_{h \to +0} \phi (\psi(t-h;t,s_{f}),t-h) ~ \text{and}~ \phi(s,T)= \lim_{h \to +0} \phi (\psi (T-h;T,s),T-h).  $$
 Also in case (A1)-(c) and case (A1)-(d), a function $p \in L^{1}(S_{Tf};Z)$ is called a weak solution to (\ref{1.1}) in the following sense
 $$ \int_{0}^{T} \int_{0}^{s_{f}} \int_{\Omega} p(s,t,x)(- D_{\psi} \phi (s,t,x)- k \Delta \phi(s,t,x)+ \mu(s,t,x) \phi(s,t,x)) dx ds dt $$
 $$ = \int_{0}^{s_{f}} \int_{\Omega} p_{0}(s,x) \phi(s,0,x) dx ds  + \int_{0}^{T} \int_{0}^{s_{f}} \int_{\Omega} f(s,t,x) \phi (s,t,x) dx ds dt $$
 where $\phi$ is any absolutely continuous function along almost every characteristic line and lies in $ \Phi_{1}$ for case (A1)-(c) and in $\Phi_{2}$ for case (A1)-(d).
 \\
  Now, question arises, what guarantees us the existence of weak solutions. For this we have the following result:

 \begin{Theorem} \label{mt4}
  The mild solution $p \in
 L^{\infty}(0,T;L^{1}(0,s_{f};Z))$ of (\ref{1.1}) is also a weak solution of (\ref{1.1}).
 \end{Theorem}
 \begin{proof} Let us assume that $p \in L^{\infty}(0,T;L^{1}(0,s_{f};Z))$ be a mild solution of (\ref{1.1}). Also assume that $\phi \in \Phi_{1}$ in case (A1)-(a)  and $\phi \in \Phi_{2}$ in case (A1)-(b). Also let us define the integral
\begin{equation}
I = \int_{0}^{T-h}\int_{0}^{s_{f}} \int_{\Omega} \frac{1}{h} \left( p(\psi(t+h;t,s),t+h,x) - p(\psi(t;t,s),t,x) \right) \phi(s,t,x) dx ds dt.
\end{equation}
Then as $h \to 0$, $I$ will converge to $$\int_{0}^{T}\int_{0}^{s_{f}} \int_{\Omega} D_{\psi}p(s,t,x) \phi(s,t,x) dx ds dt  $$
\begin{eqnarray}
&=& \int_{0}^{T}\int_{0}^{s_{f}} \int_{\Omega} \left( \Delta p(s,t,x)- \partial_{s}\gamma(s,t)p(s,t,x) - \mu(s,t,x)p(s,t,x) \right) \phi(s,t,x) dx ds dt \nonumber
\\ &+& \int_{0}^{T}\int_{0}^{s_{f}} \int_{\Omega}f(s,t,x)\phi(s,t,x) dx ds dt  \\
&=& \int_{0}^{T}\int_{0}^{s_{f}} \int_{\Omega} \left( \Delta \phi(s,t,x)- \partial_{s}\gamma(s,t) \phi(s,t,x) - \mu(s,t,x) \phi(s,t,x) \right) p(s,t,x) dx ds dt \nonumber
\\
\label{4.16} &+& \int_{0}^{T}\int_{0}^{s_{f}} \int_{\Omega}f(s,t,x) \phi(s,t,x) dx ds dt
\end{eqnarray}
Now, using the change of variables $(s,t) = (y,\zeta)$, where $(s,t)$ is related to $(y,\zeta)$ by the relation $(s,t) = (\psi(\zeta -h;\zeta, y),\zeta - h)$.
Then $I$ can be written as
\begin{eqnarray*}
I &=& \frac{1}{h} \int_{h}^{T} \int_{\psi(\zeta;\zeta -h,0)}^{\psi(\zeta;\zeta -h,s_{f})} \int_{\Omega}  p(y,\zeta,x) \phi(\psi(\zeta -h;\zeta,y),\zeta-h,x) \exp{\left( \int_{\zeta - h}^{\zeta} \partial_{s}\gamma(\psi(\eta;\zeta,y),\eta) d \eta \right)} dx dy d\zeta
\\ &=& \frac{1}{h} \int_{h}^{T-h} \int_{\psi(\zeta;\zeta -h,0)}^{\psi(\zeta;\zeta -h,s_{f})} \int_{\Omega}  p \left( \phi(\psi(\zeta -h;\zeta,y),\zeta-h,x) - \phi(y,\zeta,x) \right) \\
&\times & \exp{\left( \int_{\zeta - h}^{\zeta} \partial_{s}\gamma(\psi(\eta;\zeta,y),\eta) d \eta \right)} dx dy d\zeta \\
&+&  \frac{1}{h}\int_{h}^{T-h} \int_{\psi(\zeta;\zeta -h,0)}^{\psi(\zeta;\zeta -h,s_{f})} \int_{\Omega} p \phi(\psi(\zeta -h;\zeta,y),\zeta-h,x) \\
&\times & \left(\exp{\left( \int_{\zeta - h}^{\zeta} \partial_{s}\gamma(\psi(\eta;\zeta,y),\eta) d \eta \right)}-1 \right) dx dy d\zeta  \\ 
&+& \frac{1}{h} \int_{T-h}^{T} \int_{\psi(\zeta;\zeta -h,0)}^{\psi(\zeta;\zeta -h,s_{f})} \int_{\Omega} p \phi(\psi(\zeta -h;\zeta,y),\zeta-h,x) \exp{\left( \int_{\zeta - h}^{\zeta} \partial_{s}\gamma(\psi(\eta;\zeta,y),\eta) d \eta \right)} dx dy d\zeta \\
&-& \frac{1}{h} \int_{h}^{T-h} \int_{\psi(\zeta;\zeta -h,0)}^{\psi(\zeta;\zeta -h,s_{f})} \int_{\Omega} p(y,\zeta,x) \phi(y,\zeta,x) dx dy d\zeta \\
&-& \frac{1}{h} \int_{0}^{h} \int_{0}^{s_{f}} \int_{\Omega} p(y,\zeta,x) \phi(y,\zeta,x) dx dy d\zeta
\end{eqnarray*}
Now, combining first two integrals and taking other integrals separately, $I$ will converge to
\begin{eqnarray}
&& \int_{0}^{T}\int_{0}^{s_{f}} \int_{\Omega} p(s,t,x) \left( -D_{\psi} \phi(s,t,x) - \partial_{s}\gamma(s,t) \phi(s,t,x) \right) dx ds dt \nonumber
 \\
\label{4.17}
 &+& \int_{0}^{T} \int_{\Omega} \gamma(0,\zeta) p(0,\zeta,x) \phi(0,\zeta,x) dx d\zeta + \int_{0}^{s_{f}} \int_{\Omega} p_{0}(s,x)\phi(s,0,x) dx ds
\end{eqnarray}
 as $h \to 0$. Now, from (\ref{4.16}) and (\ref{4.17}) it is clear that $p$ is a weak solution. Similar arguments will work in case (A1)-(c) and case (A1)-(d).
\end{proof}

\mysection{Optimal Control Problem} In this section, we need to characterize tangent cones and normal cones. In \cite{MR1702849}, the following characterization is given for tangent cone and normal cone. \\
 The \textbf{Tangent cone}  $~\mathcal{T}_{u}(\mathcal{U})~$ to the convex set (\ref{n2.2}) has the following characterization:
$v_{1} \in \mathcal{T}_{u}(\mathcal{U}) $ iff almost everywhere on $\Omega_{T_{s}}$
 $$ v_{1}(s,t,x) \ge 0 ~~\text{if}~~ u(s,t,x) = \phi_{l}(s,t,x)  $$
 $$ v_{1}(s,t,x) \le 0 ~~\text{if}~~ u(s,t,x) = \phi_{m}(s,t,x). $$
 The \textbf{Normal cone} $~\mathcal{N}_{u}(\mathcal{U})~$ to the convex set (\ref{n2.2}) has the following characterization:
$v_{1} \in \mathcal{N}_{u}(\mathcal{U}) $ iff almost everywhere on $\Omega_{T_{s}}$
 $$ v_{1}(s,t,x) \ge 0 ~~\text{if}~~ u(s,t,x) = \phi_{m}(s,t,x)  $$
 $$ v_{1}(s,t,x) = 0 ~~\text{if}~~\phi_{l}(s,t,x) < u(s,t,x) < \phi_{m}(s,t,x) $$
 $$ v_{1}(s,t,x) \le 0 ~~\text{if}~~ u(s,t,x) = \phi_{l}(s,t,x). $$

Now, let us consider the following  optimal birth control problem \begin{equation} \label{5.11}
 \text{Minimize}~ J(\beta)= \text{Minimize} \int_{0}^{T} \int_{0}^{s_{f}} \int_{\Omega} \left[ p^{\beta}(s,t,x)- \frac{1}{2} \rho (\beta (s,t,x))^{2} \right] dx ds dt,
 \end{equation}
where $p^{\beta}(s,t,x)$ is the solution of (\ref{1.1}) corresponding to the birth function $\beta$. The term  $-\frac{1}{2} \rho (\beta)^{2}$ in the objective functional implies that if cost of birth control is higher, the fertility rate will be lower which also matches with our intuition. \\
In this section we will assume that the  assumptions (A1)-(A5) hold and also the assumptions of (Theorem \ref{mt3}) hold.
\begin{Definition}
 A pair $(\beta,p^{\beta})$ is said be optimal pair for the birth control problem (\ref{5.11}) if the following conditions are satisfied  \\
(1) $\beta \in \mathcal{U}$. \\
(2) $p^{\beta}$ solves the system(\ref{1.1}). \\
(3) $(\beta,p^{\beta})$ minimizes the functional $J(\beta)$.
\end{Definition}

Let us define the functional $\Psi$ in $\mathcal{U}$ as follows
\begin{equation} \label{5.12}
\Psi (\beta) =
\begin{cases}
\int_{0}^{T} \int_{0}^{s_{f}} \int_{\Omega} \left[ p^{\beta}(s,t,x)- \frac{1}{2} \rho (\beta (s,t,x))^{2} \right] dx ds dt, ~~ \beta \in \mathcal{U} \\
+\infty, ~~ \text{otherwise}.
\end{cases}
\end{equation}

Our first task is to prove that $\Psi$ is lower semicontinuous.
\begin{Lemma} \label{ml1}
 The functional $\Psi$ is lower semicontinuous.
 \end{Lemma}

\begin{proof}
 Let $\beta_{n}$ be a sequence in $L^{1}(\Omega_{Ts})$ converges to $\beta$. Without loss of generality assume that $\beta_{n} \in \mathcal{U}$. Since $p^{\beta_{n}}$ is bounded
(Theorem\ref{mt3}) and the bound depends upon $\beta_{n}$, we can find a subsequence $\beta_{n} (s,t,x) \rightarrow \beta(s,t,x)$ and $p^{\beta_{n}}(s,t,x) \rightarrow p^{\beta}(s,t,x)$ in $\Omega_{Ts}$. So, we have
$$p^{\beta_{n}}(s,t,x) + \frac{1}{2} \rho (\beta_{n}(s,t,x))^{2} \rightarrow p^{\beta}(s,t,x) +  \frac{1}{2} \rho (\beta (s,t,x))^{2}. $$
Now, using the Lebesgue dominated convergence theorem, we can conclude that
$$ \int_{0}^{T} \int_{0}^{s_{f}} \int_{\Omega} \left[ p^{\beta_{n}}(s,t,x)- \frac{1}{2} \rho (\beta_{n} (s,t,x))^{2} \right] dx ds dt  $$
converges to $$\int_{0}^{T} \int_{0}^{s_{f}} \int_{\Omega} \left[ p^{\beta}(s,t,x)- \frac{1}{2} \rho (\beta (s,t,x))^{2} \right] dx ds dt $$
Therefore, by Fatou's lemma,
\begin{dmath*}
 \liminf_{n \rightarrow \infty} \int_{0}^{T} \int_{0}^{s_{f}} \int_{\Omega} \left[ p^{\beta_{n}}(s,t,x)- \frac{1}{2} \rho (\beta_{n} (s,t,x))^{2} \right] dx ds dt \\ \ge \int_{0}^{T} \int_{0}^{s_{f}} \int_{\Omega} \left[ p^{\beta}(s,t,x)- \frac{1}{2} \rho (\beta (s,t,x))^{2} \right] dx ds dt ,
\end{dmath*} 
  which is
$$ \liminf_{n \rightarrow \infty} \Psi (\beta_{n}) \ge \Psi (\beta).$$ Therefore
$\Psi $ is lower semicontinuous.
\end{proof}

\begin{Theorem} \label{mt5}
 Let $(\beta_{*},p^{\beta_{*}})$ be optimal pair for the least cost size problem and $\phi^{\beta_{*}}$ be the solution to the adjoint system (\ref{4.5}) in case (A1)-(a), (A1)-(c) and adjoint system (\ref{4.6}) in case (A1)-(b), (A1)-(d) respectively. Then for $f^{*}(s,t,x)=-c$, where $c$ is a positive constant
$$ \beta_{*}(s,t,x) = F \left( -\frac{r(s,t,x) \phi^{\beta_{*}}(0,t,x) p^{\beta_{*}}(s,t,x)}{c \rho} \right) $$
$F \colon L^{1}(\Omega_{Ts}) \mapsto L^{\infty}(\Omega_{Ts})$ is defined by
\begin{equation}
(Fh)(s,t,x)=
\begin{cases}
\phi_{l}, ~~ \text{if}~ h(s,t,x) < \phi_{l} \\
h(s,t,x), ~~ \text{if} ~ \phi_{l} \le h(s,t,x) \le \phi_{m} \\
\phi_{m}, ~~ \text{if} ~ h(s,t,x) > \phi_{m}.
\end{cases}
\end{equation}
\end{Theorem}

 \begin{proof}
Let $z_{\epsilon} = \frac{1}{\epsilon} [p^{\beta_{*}+\epsilon \delta_{*}} - p^{\beta_{*}} ],$ where $ \delta_{*}=\beta - \beta_{*}$, then $z_{\epsilon}$ converges to $z \in L^{\infty}(S_{Tf} ; Z)$ and $z$ satisfies \\
  \begin{equation} \label{5.2}
 \begin{cases}
  D_{\psi}z(s,t) = A z(s,t) - \partial_{s}\gamma(s,t)z(s,t) - M (s,t)z(s,t)  \quad \text{a.e} \quad (s,t) \in S_{T_{f}} \\
 \gamma(0,t)z(0,t) =  \int_{s_{i}}^{s_{f}} r(s,t)B_{*}(s,t)z(s,t)ds + \int_{0}^{s_{f}}\delta_{*} p^{\beta_{*}}(s,t) ds, \quad \text{a.e} \quad t \in (0,T) \\
 z(s,0 = 0, \quad \text{a.e} \quad s \in (0,s_{f}),
 \end{cases}
\end{equation}
where $\delta^{*}p^{\beta^{*}}(s,t)= \delta^{*}(s,t,\cdot)p^{\beta^{*}}(s,t,\cdot)$ and also $\delta^{*}p^{\beta^{*}} \in L^{\infty}(S_{Tf};Z)$.
\\
Because $z(s,t)$ is a weak solution of (\ref{5.2}), we have
\\
 $$ \int_{0}^{T} \int_{0}^{s_{f}} \int_{\Omega} z(s,t,x)(- D_{\psi} \phi (s,t,x)- k \Delta \phi(s,t,x)+ \mu(s,t,x) \phi(s,t,x) - r(s,t,x)\beta (s,t,x) \phi(0,t,x)) dx ds dt $$
$$ = \int_{0}^{T} \int_{0}^{s_{f}} \int_{\Omega} \delta_{*}(s,t,x) p^{\beta_{*}}(s,t,x) \phi(s,t,x) dx ds dt. $$
Because $\phi(s,t)$ satisfies \\
$$  D_{\psi} \phi (s,t)= - A^{*} \phi(s,t)- M^{*}(s,t) \phi(s,t) + r(s,t)B^{*} (s,t) \phi(0,t)+f^{*}(s,t)-1, $$
we have
\begin{equation} \label{5.4}
 \int_{0}^{T} \int_{0}^{s_{f}} \int_{\Omega} z(s,t,x)(f^{*}(s,t,x)-1) dx ds dt = \int_{0}^{T} \int_{0}^{s_{f}} \int_{\Omega} \delta_{*}r^{*}(s,t,x) p^{\beta_{*}}(s,t,x) \phi(0,t,x) dx ds dt.
 \end{equation}
 Because $(\beta_{*},p^{\beta_{*}})$ is optimal pair, we have
 $$ \lim_{\epsilon \to +0} \frac{1}{\epsilon}[\Psi(\beta_{*}+\epsilon \delta_{*})-\Psi(\beta_{*})] \ge 0,$$ which implies
 $$ \lim_{\epsilon \to +0} \int_{0}^{T} \int_{0}^{s_{f}} \int_{\Omega} \left[ \frac{p^{\beta_{*}+\epsilon \delta_{*}}-p^{\beta_{*}}}{\epsilon} - \frac{\rho \epsilon (\delta_{*})^{2}}{2} - \rho \beta_{*} \delta_{*}\right] dx ds dt \ge 0,$$ which further implies
 $$\int_{0}^{T} \int_{0}^{s_{f}} \int_{\Omega} \left( z(s,t,x)-\rho \beta_{*} \delta_{*} \right) dx ds dt \ge 0 .$$
 For $f^{*}(s,t,x)=-c+1$, where $c$ is a nonnegative constant not equal to $1$ and using (\ref{5.4}), we get
 $$ \int_{0}^{T} \int_{0}^{s_{f}} \int_{\Omega} (\beta_{*}-\beta) \left( \rho \beta_{*}+ \frac{r(s,t,x)p^{\beta_{*}}(s,t,x)\phi^{\beta_{*}}(0,t,x)}{c} \right) dx ds dt \ge 0,$$ which implies
  $$ \int_{0}^{T} \int_{0}^{s_{f}} \int_{\Omega} \rho (\beta_{*}-\beta) \left( \beta_{*}+ \frac{r(s,t,x)p^{\beta_{*}}(s,t,x)\phi^{\beta_{*}}(0,t,x)}{c \rho} \right) dx ds dt \ge 0 .$$
  Let $$ \mathcal{V}_{1}=\left \{ (s,t,x) \in \Omega_{Ts} ~|~ -\frac{r(s,t,x)p^{\beta_{*}}(s,t,x)\phi^{\beta_{*}}(0,t,x)}{c \rho} < \phi_{l}  \right \}. $$
  Choose $\beta = \phi_{l}$  on $\mathcal{V}_{1}$.
So, in this case  $$ \int_{\mathcal{V}_{1}} \rho (\beta_{*}-\phi_{l}) \left( \beta_{*}+ \frac{r(s,t,x)p^{\beta_{*}}(s,t,x)\phi^{\beta_{*}}(0,t,x)}{c \rho} \right) dx ds dt =0, $$
 because $\rho$ is a positive constant, we have $\beta_{*}=\phi_{l}.$ \\

Similarly, let  $$ \mathcal{V}_{2}= \left \{ (s,t,x) \in \Omega_{Ts} ~|~ -\frac{r(s,t,x)p^{\beta_{*}}(s,t,x)\phi^{\beta_{*}}(0,t,x)}{c \rho} > \phi_{m} \right \}$$ and choose $\beta = \phi_{m}$ on $\mathcal{V}_{2}$.
Then, we have $$ \int_{\mathcal{V}_{2}} \rho (\beta_{*}-\phi_{m}) \left( \beta_{*}+ \frac{r(s,t,x)p^{\beta_{*}}(s,t,x)\phi^{\beta_{*}}(0,t,x)}{c \rho} \right) dx ds dt =0.$$
Since $\rho$ is a positive constant, we have $\beta_{*}=\phi_{m}.$
\\
If $$\mathcal{V}_{3}= \left \{  (s,t,x) \in \Omega_{Ts} ~|~\phi_{l} \le -\frac{r(s,t,x)p^{\beta_{*}}(s,t,x)\phi^{\beta_{*}}(0,t,x)}{c \rho} \le \phi_{m} \right \},$$
then on $\mathcal{V}_{3}$ $$ \beta_{*}=-\frac{r(s,t,x) \phi^{\beta_{*}}(0,t,x) p^{\beta_{*}}(s,t,x)}{c \rho}. $$
\end{proof}

 \begin{Theorem}
Suppose that the assumptions $(A1)-(A5)$
hold and assumptions of theorem \ref{mt5} hold with $$\frac{M_{1}M_{4}+M_{2}M_{3}}{c \rho} <1,$$
 then the optimal control problem (\ref{5.11}) has a unique optimal control $\bar{\beta} \in \Omega_{T_{s}}$ where $M_{3}$ and $M_{4}$ are supremum of $|p|$ and $|\phi|$ respectively.
 \end{Theorem}
 \begin{proof}
 We know that the functional $\Psi$ defined in (\ref{5.12}) is lower semicontinuous, so using Ekeland's variational principle, for given $\epsilon >0,$ there exists $\beta_{\epsilon} \in \mathcal{U}$ such that
 \begin{eqnarray} \label{ek2}
 \Psi({\beta_{\epsilon}}) &\le& \inf_{\beta \in \mathcal{U}} \Psi(\beta) + \epsilon
 \\
 \Psi({\beta_{\epsilon}}) &\le& \inf_{\beta \in \mathcal{U}} \{ \Psi(\beta) + \sqrt{\epsilon} \| \beta - \beta_{\epsilon} \|_{L^{1}(\Omega_{T_{s}})} \}.
 \end{eqnarray}
 Thus the perturbed functional
 \begin{eqnarray}
 \Psi_{\epsilon}(\beta) = \Psi(\beta) + \sqrt{\epsilon} \| \beta - \beta_{\epsilon} \|_{L^{1}(\Omega_{T_{s}})}
 \end{eqnarray}
 attain its infimum at $\beta_{\epsilon}$.
 Therefore,
 $$ \lim_{\epsilon ' \to 0} \frac{1}{\epsilon '} \left[ \Psi_{\epsilon}(\beta_{\epsilon}+ \epsilon' \delta_{*}) - \Psi_{\epsilon}(\beta_{\epsilon}) \right] \ge 0,  $$ which implies
 $$ \lim_{\epsilon ' \to 0} \frac{1}{\epsilon '} \left[ \Psi(\beta_{\epsilon}+ \epsilon' \delta_{*}) + \sqrt{\epsilon} \| \epsilon' \delta_{*} \| -\Psi(\beta_{\epsilon} \right] \ge 0.$$
 Following the same steps as in theorem \ref{mt5}, we get
 \begin{equation} \label{ek1}
 \int_{0}^{T} \int_{0}^{s_{f}} \int_{\Omega} \rho \delta_{*} \left( \beta_{\epsilon}+ \frac{r(s,t,x)p^{\beta_{\epsilon}}(s,t,x)\phi^{\beta_{\epsilon}}(0,t,x)}{c \rho} \right) dx ds dt + \sqrt{\epsilon} \int_{0}^{T} \int_{0}^{s_{f}} \int_{\Omega} | \delta_{*} | dx ds dt \ge 0.
 \end{equation}
 We know that $v_{1}$ lies in tangent cone $\mathcal{T}_{u}(\mathcal{U})$ iff almost everywhere on $\Omega_{T_{s}}$
 $$ v_{1}(s,t,x) \ge 0 ~\text{if}~ u(s,t,x) = \phi_{l}(s,t,x)  $$
 $$ v_{1}(s,t,x) \le 0 ~\text{if}~ u(s,t,x) = \phi_{m}(s,t,x). $$
 Thus, $\delta_{*} \in \mathcal{T}_{u}(\mathcal{U})$ and also $\delta_{*} = \beta - \beta_{\epsilon}$ depends on $\beta$ which lies in $\mathcal{U}$, which means (\ref{ek1}) holds for any $\delta_{*} \in \mathcal{T}_{u}(\mathcal{U})$.
 Therefore using the structure of normal cones, there exists $\theta(s,t,x) \in L^{\infty}(\Omega_{T_{s}}), \ \| \theta \|_{\infty} \le 1~$ such that
 $$ \rho \beta_{\epsilon}+ \frac{r(s,t,x)p^{\beta_{\epsilon}}(s,t,x)\phi^{\beta_{\epsilon}}(0,t,x)}{c} + \sqrt{\epsilon} \theta \in \mathcal{N}_{\beta_{\epsilon}}(\mathcal{U}).$$
 Therefore
 $$
 \beta_{\epsilon}(s,t,x) = F \left( -\frac{r(s,t,x)p^{\beta_{\epsilon}}(s,t,x)\phi^{\beta_{\epsilon}}(0,t,x)}{c \rho} + \frac{\sqrt{\epsilon} \theta}{\rho}  \right).
 $$
 Now, we show the uniqueness of optimal control. Define
 $$ G \colon \mathcal{U} \subset L^{\infty}(\Omega_{T_{s}}) \mapsto \mathcal{U}  $$ by
  $$(G \beta)(s,t,x) =  F \left( -\frac{r(s,t,x)p^{\beta}(s,t,x)\phi^{\beta}(0,t,x)}{c \rho} + \frac{\sqrt{\epsilon} \theta}{\rho}  \right). $$
 Then for any $(s,t,x) \in \Omega_{T_{s}}$, we have
\begin{eqnarray*}
 && |(G \beta_{1})(s,t,x) - (G \beta_{2})(s,t,x) | \\ &=&  F \left( -\frac{r(s,t,x)p^{\beta_{1}}(s,t,x)\phi^{\beta_{1}}(0,t,x)}{c \rho} +\frac{\sqrt{\epsilon} \theta}{\rho} \right) - F \left( -\frac{r(s,t,x)p^{\beta_{2}}(s,t,x)\phi^{\beta_{2}}(0,t,x)}{c \rho} + \frac{\sqrt{\epsilon} \theta}{\rho} \right)
  \\ & \le & \left| \frac{r(s,t,x)}{c \rho} \right| \left| \left( p^{\beta_{1}}(s,t,x) \phi^{\beta_{1}}(0,t,x)- p^{\beta_{2}}(s,t,x) \phi^{\beta_{2}}(0,t,x)\right) \right| \\
 &\le & \left| \frac{r(s,t,x)}{c \rho} \right|  \left( | p^{\beta_{1}}(s,t,x)-p^{\beta_{2}}(s,t,x) | | \phi^{\beta_{1}}(0,t,x) |+| p^{\beta_{2}}(s,t,x)| | \phi^{\beta_{1}}(0,t,x)-\phi^{\beta_{2}}(0,t,x)|\right) \\
 &\le & \frac{1}{c \rho} (M_{1}M_{4}+M_{2}M_{3}) \| \beta_{1}- \beta_{2} \|_{L^{\infty}(\Omega_{T_{s}})},
 \end{eqnarray*}
 because
 $$\frac{1}{c \rho} (M_{1}M_{4}+M_{2}M_{3}) < 1.$$
 Hence $G$ is a contraction mapping, so have a unique fixed point $\bar{\beta} \in \mathcal{U}$. Optimality condition in theorem \ref{mt5} gives us the uniqueness of the optimal control.  Now, we prove the existence of optimal controller. Let
 \begin{eqnarray}
 \Psi(\bar{\beta}) = \inf \left\{ \Psi(\beta) : \beta \in \mathcal{U} \right\}.
 \end{eqnarray}
 Computing the norm, we obtain
 \begin{eqnarray*}
 && \| G\beta_{\epsilon} - \beta_{\epsilon} \|_{L^{\infty}(\Omega_{T_{s}})} \\ && =  F \left( -\frac{r(s,t,x)p^{\beta_{\epsilon}}(s,t,x)\phi^{\beta_{\epsilon}}(0,t,x)}{c \rho} + \frac{\sqrt{\epsilon} \theta}{\rho}  \right)-F \left( -\frac{r(s,t,x)p^{\beta_{\epsilon}}(s,t,x)\phi^{\beta_{\epsilon}}(0,t,x)}{c \rho} \right)  \\ && \le \frac{\sqrt{\epsilon}}{\rho}.
 \end{eqnarray*}
 Therefore
 \begin{eqnarray*}
 \| \bar{\beta} - \beta_{\epsilon} \|_{L^{\infty}(\Omega_{T_{s}})} &=& \| G\bar{\beta} - \beta_{\epsilon} \|_{L^{\infty}(\Omega_{T_{s}})} \\
&\le & \| G\bar{\beta} - G\beta_{\epsilon} \|_{L^{\infty}(\Omega_{T_{s}})} + \|  G \beta_{\epsilon} -\beta_{\epsilon} \|_{L^{\infty}(\Omega_{T_{s}})} \\
&\le & \frac{1}{c \rho} (M_{1}M_{4}+M_{2}M_{3}) \| \bar{\beta} - \beta_{\epsilon} \|_{L^{\infty}(\Omega_{T_{s}})} + \frac{\sqrt{\epsilon}}{\rho}.
 \end{eqnarray*}
Which further gives
 $$
  \| \bar{\beta} - \beta_{\epsilon} \|_{L^{\infty}(\Omega_{T_{s}})} \le \frac{1}{\rho} \left( 1 -  \frac{1}{c \rho} (M_{1}M_{4}+M_{2}M_{3}) \right)^{-1} \sqrt{\epsilon},
$$
 which implies $\beta_{\epsilon} \to \bar{\beta} ~ \text{in}~ L^{\infty}(\Omega_{T_{s}}) ~ \text{as}~ \epsilon \to 0^{+}. $
 Hence by (\ref{ek2}), we have
 $$ \Psi(\bar{\beta}) = \inf_{\beta \in \mathcal{U}} \Psi(\beta). $$
 \end{proof}

\vspace{.5cm}


\bibliographystyle{amsplain}	


\begin{thebibliography}{10}
\bibitem{Lotka1}
Sharpe, F.,~Lotka, A.J.
\newblock A problem in age-distribution.
\newblock {\em Philos. Mag.}, 6:435-438, (1911).

\bibitem{Lotka2}
Lotka, A.J.
\newblock The stability of the age distribution.
\newblock {\em Proc. Natl. Acad. Sci.}, 8:330-345, (1922).

\bibitem{auger2008structured}  Webb,G. F. Population models structured by age, size, and spatial position. In Structured population
models in biology and epidemiology, volume 1936 of Lecture Notes in Math., pages 1-49. Springer,
Berlin, (2008).

\bibitem{Kato2016}
Kato, N.
\newblock Abstract linear partial differential equations related to
  size-structured population models with diffusion.
\newblock {\em J. Math. Anal. Appl.}, 436(2):890--910, (2016).

\bibitem{MR1797596}
Anita, S.
\newblock {\em Analysis and control of age-dependent population dynamics},
  volume~11 of {\em Mathematical Modelling: Theory and Applications}.
\newblock Kluwer Academic Publishers, Dordrecht, (2000).

\bibitem{MR1702849}
Barbu, V.,~Iannelli, M.
\newblock Optimal control of population dynamics.
\newblock {\em J. Optim. Theory Appl.}, 102(1):1--14, (1999).


\bibitem{MR3595204}
Liu, R.,~Liu, G.
\newblock Optimal birth control problems for a nonlinear vermin population
  model with size-structure.
\newblock {\em J. Math. Anal. Appl.}, 449(1):265--291, (2017).

\bibitem{MR2863964}
He, Z. R.,~Liu, Y.
\newblock An optimal birth control problem for a dynamical population model
  with size-structure.
\newblock {\em Nonlinear Anal. Real World Appl.}, 13(3):1369--1378, (2012).

\bibitem{MR1027051}
Chan, W.~L.,~Zhu, G. B..
\newblock Optimal birth control of population dynamics.
\newblock {\em J. Math. Anal. Appl.}, 144(2):532--552, (1989).


\bibitem{MR1043120}
Chan, W. L.,~Guo, B. Z.
\newblock Optimal birth control of population dynamics. {II}. {P}roblems with
  free final time, phase constraints, and mini-max costs.
\newblock {\em J. Math. Anal. Appl.}, 146(2):523--539, (1990).




%

%
\bibitem{NRN1}
Luo, Z.
\newblock  Optimal birth control for an age-dependent competition system of n species.
\newblock {\em J. Syst. Sci.Complex.}, 20(3), 403-415., (2007).

\bibitem{NRN2}
Chan, W. L., Zhu, G. B.
\newblock  Pareto Optimal Birth Control of Age-Dependent Populations.
\newblock {\em IFAC Proc. Vol.}, 22(4), 147-151., (1989).

\bibitem{NRN3}
Luo, Z.
\newblock  Optimal birth control for competition system of three species with age-structure.
\newblock {\em J. Appl. Math. Comput.}, 24(1), 49-64., (2007).


\bibitem{NRN4}
Luo, Z., He, Z. R.,  Li, W. T.
\newblock  Optimal birth control for predator–prey system of three species with age-structure.
\newblock {\em Appl. Math. Comput.}, 155(3), 665-685., (2004).

\bibitem{NRN5}
Zhao, C., Wang, M. S., Zhao, P.,  Wan, A. H.
\newblock  Optimal birth control for periodic age-dependent population dynamics.
\newblock {\em In Proceedings of 2004 International Conference on Machine Learning and Cybernetics (IEEE Cat. No. 04EX826).}, (Vol. 1, pp. 383-386)., (2004).

\bibitem{NRN6}
He, Z. R., Wang, M. S.,  Ma, Z. E.
\newblock  Optimal birth control problems for nonlinear age-structured population dynamics.
\newblock {\em Discrete Cont. Dyn. Sys. ser. B.}, 4(3), 589., (2004).

\bibitem{NRN8}
Huang, Y.,  Zhao, Y.
\newblock  Optimal Birth Control of A Nonlinear Population Diffusion with External Constraint.
\newblock {\em Control Theory Appl.},  05,534–544., (1994).

\bibitem{MR3703602}
Kato, N.
\newblock Optimal harvesting for linear size-structured population models with
  diffusion.
\newblock {\em J. Nonlinear Convex Anal.}, 18(7):1335--1348, (2017).

\bibitem{NRN9}
Kato, N.
\newblock Optimal harvesting for nonlinear size-structured population dynamics.
\newblock {\em J. Math. Anal. Appl.}, 342(2), 1388-1398, (2008).

\bibitem{MR2520362}
Iannelli, M.,~Marinoschi, G.
\newblock Harvesting control for an age-structured population in a multilayered
  habitat.
\newblock {\em J. Optim. Theory Appl.}, 142(1):107--124, (2009).

\bibitem{MR937160}
Getz, W. M.
\newblock Harvesting discrete nonlinear age and stage structured populations.
\newblock {\em J. Optim. Theory Appl.}, 57(1):69--83, (1988).


\bibitem{NRN7}
Brokate, M.
\newblock   Pontryagin's principle for control problems in age-dependent population dynamics.
\newblock {\em J. Math. Biol.}, 23(1), 75-101., (1985).

\bibitem{RN29}
Morozov, A.,~Kuzenkov, O. A.,~Arashkevich, E. G.
\newblock Modelling optimal behavioural strategies in structured populations
  using a novel theoretical framework.
\newblock {\em Scientific Reports}, 9, (2019).

\bibitem{NRN10}
Ekeland, I.
\newblock  On the variational principle.
\newblock {\em J. Math. Anal. Appl.}, 47(2), 324-353., (1974).

\end{thebibliography}
\end{document}